\documentclass[11pt]{amsart}

\usepackage{amssymb}
\usepackage{enumerate}
\usepackage{amsfonts}
\usepackage{amsmath}
\usepackage{mathrsfs}
\allowdisplaybreaks[1]

\newcounter{fig}

\newcommand{\myself}{\author{Gianluca Cassese}
                     \address{Universit\`{a} Milano Bicocca and University of Lugano}
                     \email{gianluca.cassese@unimib.it}
                     \curraddr{Department of Economics, Statistics and Management, 
                                  Building U7, Room 2097, via Bicocca 
                                  degli Arcimboldi 8, 20126 Milano - Italy}}

\newtheorem{theorem}{Theorem}
\theoremstyle{plain}

\newtheorem{corollary}{Corollary}

\newtheorem{lemma}{Lemma}

\newcommand{\Sim}{\mathscr{S}}

\newcommand{\A}{\mathscr{A}}

\newcommand{\N}{\mathbb{N}}

\newcommand{\abs}[1]{\vert #1\vert}

\newcommand{\net}[3]{\langle #1_{#2}\rangle_{#2\in #3} } 
\newcommand{\neta}[1]{\net{#1}{\alpha}{\mathfrak A}} 
\newcommand{\nnet}[3]{\langle #1\rangle_{#2\in #3} } 
 
\newcommand{\seq}[2]{\net{#1}{#2}{\mathbb{N}}} 
\newcommand{\sseq}[2]{\nnet{#1}{#2}{\mathbb{N}}} 
\newcommand{\seqn}[1]{\seq{#1}{n}} 
\newcommand{\sseqn}[1]{\sseq{#1}{n}}

\newcommand{\norm}[1]{\Vert #1\Vert}

\newcommand{\set}[1]{\mathbf{1}_{#1}}

\newcommand{\iref}[1]{(\textit{\ref{#1}})}
\newcommand{\imply}[2]{\iref{#1}$\Rightarrow$\iref{#2}}

\newcommand{\K}{\mathcal K}
\newcommand{\Bl}{\mathscr B(\lambda)}
\newcommand{\M}{\mathscr M}
\newcommand{\Ha}{\mathscr H}

\newcommand{\lp}{\lambda_\M^\perp}
\newcommand{\lc}{\lambda_\M^c}
\newcommand{\Av}[1]{\mathbf A(#1)}
\newcommand{\AM}{\mathbf A(\M)}
\newcommand{\LM}{\mathbf L(\M)}

\newcommand{\cl}[2]{\overline{#1}^{#2}}
\newcommand{\C}{\mathcal C}

\begin{document}

\title[Halmos Savage Theorem]{The Theorem of Halmos and Savage under 
Finite Additivity}
\myself
\date
\today


\maketitle

\begin{abstract}
Given a generalization of Lebesgue decomposition we obtain an extension to
the finitely additive setting of the theorems of Halmos and Savage and of Yan.
\end{abstract}

\section{Introduction and Notation}
In a paper that soon became a classic in statistics \cite{halmos savage}, Halmos 
and Savage illustrated the powerful implications of the Radon Nikodym theorem for
the theory of sufficient statistics. One of their results, Lemma 7, deals with dominated
sets of probability measures and states that each such set admits an equivalent,
countable subset. This lemma rapidly obtained its own popularity, proving to be 
very useful in a variety of different contexts, such as the proof of Yan Theorem, 
another classical result in probability and in mathematical finance.

In their proof, Halmos and Savage exploit extensively countable additivity
and the fact that the underlying family is a $\sigma$-algebra. Both
properties are essential as they allow, loosely speaking, for the possibility of
taking limits. For this reason their method of proof cannot be adapted to 
the case in which probability is just \textit{finitely} additive, a situation of 
interest for the subjective theory of probability originating from the seminal 
work of de Finetti \cite{de finetti} and, more generally, for decision theory 
in which countable additivity is more an exception than a rule. Finite additivity
is also unavoidable in many classical problems in which it is needed to take
extensions of the given set function.

In this short note we extend the original result of Halmos and Savage to the
case of finitely additive probability measures and obtain, as a corollary, an
analogous extension of the theorem of Yan \cite{yan}. The proof is,
somehow surprisingly, straightforward and does not make use but of
classical decomposition results of set functions, ultimately due to Bochner
and Phillips.

In the following, $\Omega$ will be a fixed, nonempty set and $\A$ an algebra 
of subsets of $\Omega$. Also given is a positive, additive, bounded set
function $\lambda\in ba(\A)_+$. A set $\M\subset ba(\A)$ is said to be
dominated by $\lambda$ if $\mu\ll\lambda$ for every $\mu\in\M$ (in symbols
$\M\ll\lambda$). For the theory of finitely additive measures and integrals we 
mainly borrow notation, definitions and terminology from Dunford and Schwartz 
\cite{bible}, although we prefer the symbol $\abs\mu$ to denote the total variation 
measure generated by $\mu$ and we write $\mu_f$ to denote that element of
$ba(\A)$ defined implicitly by letting
\begin{equation}
\mu_f(A)=\int\set Afd\mu
\qquad A\in\A
\end{equation}
whenever $f\in L^1(\mu)$. We often write $\mu(f)$ rather than $\int fd\mu$.

The lattice symbol $\perp$ is used to define the orthogonal complement
\begin{equation*}
\M^\perp=\{\nu\in ba(\A):\nu\perp\mu\text{ for every }\mu\in\M\}
\end{equation*}
of $\M$ which is known to be a normal sublattice of $ba(\A)$, see e.g. 
\cite[1.5.6 and 1.5.8]{rao}.

\section{A Decomposition}
\label{sec lebesgue}
We associate with $\M\subset ba(\A)$ the collections
\begin{equation}
\label{AM}
\AM=\left\{\sum_n\alpha_n\frac{\abs{\mu_n}}{1\vee\norm{\mu_n}}:\ 
\mu_n\in\M,\ \alpha_n\ge0\text{ for }n=1,2,\ldots,\sum_n\alpha_n=1\right\}
\end{equation}
\begin{equation}
\label{L(M)}
\LM
=
\{\nu\in ba(\A):\nu\ll m\text{ for some }m\in\Av\M\}
\end{equation}

To obtain a simple generalization of Lebesgue decomposition, we start remarking 
that $\LM$ is a normal sublattice of $ba(\A)$ and so that, by Riesz decomposition 
Theorem \cite[1.5.10]{rao}, $ba(\A)=\LM+\LM^\perp$. To see this, take an 
increasing net $\neta\nu$ in $\LM$ with $\nu=\lim_\alpha\nu_\alpha\in ba(\A)$, 
extract a sequence $\sseqn{\nu_{\alpha_n}}$ 
such that $\norm{\nu-\nu_{\alpha_n}}=(\nu-\nu_{\alpha_n})(\Omega)<2^{-n-1}$, 
choose $m_n\in\Av\M$ such that $m_n\gg\nu_{\alpha_n}$ and define 
$m=\sum_n2^{-n}m_n\in\Av\M$. Since $m\gg m_n\gg\nu_{\alpha_n}$ for each 
$n\in\N$, there is $\delta_n>0$ such that $m(A)<\delta_n$ implies 
$\abs{\nu_{\alpha_n}}(A)<2^{-n-1}$ and, therefore, 
$\abs\nu(A)\le\abs{\nu_{\alpha_n}}(A)+2^{-n-1}\le2^{-n}$. This proves
that if $\{\nu_\alpha:\alpha\in\mathfrak A\}$ is a nonempty family in $\LM$ and
if $\bigvee_{\alpha\in\mathfrak A}\nu_\alpha$ exists in $ba(\A)$, then necessarily
$\bigvee_{\alpha\in\mathfrak A}\nu_\alpha\in\LM$. Moreover, 
$\abs{\nu_1}\le\abs\nu$ and $\nu\in\LM$ imply $\nu_1\in\LM$. Noting that 
$\LM^\perp=\AM^\perp$ we  obtain the following:

\begin{lemma}
\label{lemma lebesgue}
For each $\lambda\in ba(\A)$ and $\M\subset ba(\A)$ there is a unique way 
of writing
\begin{equation}
\label{lebesgue}
\lambda=\lc+\lp
\end{equation}
with $\lc\in\LM$ and $\lp\perp\AM$. If $\lambda$ is positive or countably additive 
then so are $\lp$ and $\lc$.
\end{lemma}

\section{The Halmos-Savage Theorem and its Implications}
\label{sec halmos savage}

We now prove the main result of the paper. Let us mention that dominated sets
of measures arise whenever dealing with a model, a statistical model e.g., in
which it is posited the existence of a reference probability measure.

\begin{theorem}[Halmos and Savage]
\label{th halmos savage}
$\M\subset ba(\A)$ is dominated if and only if $\M\ll m$ for some $m\in\Av\M$.
\end{theorem}

\begin{proof}
$\lambda$ dominates $\M$ if and only if $\lc$ does. In fact, choose $\mu\in\M$ 
and $\varepsilon>0$ and let $\delta$ be such that $\lambda(A)<\delta$ 
implies $\abs\mu(A)<\varepsilon$. Pick $B\in\A$ such that 
$\abs\mu(B^c)+\lambda_\M^\perp(B)<(\delta/2)\wedge\varepsilon$. 
Then $\lambda_\M^c(A)<\delta/2$ implies $\lambda(A\cap B)<\delta$ 
and thus
$
\abs\mu(A)\le\abs\mu(A\cap B)+\varepsilon\le2\varepsilon
$.
Lemma \ref{lemma lebesgue} proves the claim.
\end{proof}

To rephrase the above Theorem in the language of Halmos and Savage, observe
that if $\M_0=\{\mu_1,\mu_2,\ldots\}$ is the subfamily of $\M$ generating
$m=\sum_n2^{-n}\abs{\mu_n}/(1\vee\norm{\mu_n})$ and $\seq Ak$ is a sequence
in $\A$, then $\lim_k\abs{\mu_n}(A_k)=0$ for $n=1,2,\ldots$ if and only if
$\lim_k\abs{\mu}(A_k)=0$ for all $\mu\in\M$ and $\M_0$ may then be said
to be \textit{equivalent} to $\M$.

A typical application is the following:

\begin{corollary}
\label{cor drewnowski}
Let $\Ha\subset\A$ and $
\A_\Ha
=
\left\{A\in\A:\inf_{\{\alpha\subset\Ha,\ \alpha\text{ finite}\}}
\lambda\left(A\backslash\bigcup_{\alpha}H\right)=0\right\}
$. 
There exists $H_1,H_2,\ldots\in\Ha$ such
that
\begin{equation}
\label{limit}
\lim_k\sup_{A\in\A_\Ha}\lambda\left(A\backslash\bigcup_{n\le k}H_n\right)=0
\end{equation}
\end{corollary}

\begin{proof}
Write $\M=\{\lambda_H:H\in\Ha\}$ and choose 
$m=\sum_n\alpha_n\lambda_{H_n}\in\Av\M$ to be such that $m\gg\M$. 
By construction, for each $H\in\Ha$, we conclude 
$
\lim_k\lambda(H\backslash\bigcup_{n=1}^k H_n)
=
\lim_k\lambda_H(\bigcap_{n=1}^k H_n^c)
=
0
$. 
Consider a disjoint union $B=\bigcup_{j=1}^IA_j\cap K_j$ with $A_j\in\A$ and 
$K_j\in\Ha$ for $j=1,\ldots,I$ and denote by $\Ha_1$ the corresponding class. 
But then, since $\lambda_{K_j}\in\M$ for $j=1,\ldots,I$,
\begin{align*}
\lambda\left(B\cap\bigcap_{n\le k} H_n^c\right)
&=
\sum_{j=1}^I
\lambda\left(A_j\cap K_j\cap\bigcap_{n\le k} H_n^c\right)\\
&=
\lim_r\sum_{j=1}^I
\lambda\left(A_j\cap K_j\cap\bigcap_{n\le k} H_n^c\cap\bigcup_{n\le r} H_n\right)\\
&\le
\lim_r\lambda\left(\bigcap_{n\le k} H_n^c\cap\bigcup_{n\le r} H_n\right)
\end{align*}
so that $\lim_k\sup_{B\in\Ha_1}\lambda(B\cap\bigcap_{n\le k} H_n^c)=0$. Let now
$A\in\A_\Ha$. We have
\begin{align*}
\lim_k\sup_{A\in\A_\Ha}\lambda\left(A\cap\bigcap_{n\le k} H_n^c\right)
&=
\lim_k\sup_{A\in\A_\Ha}\sup_{K_1,\ldots,K_I\in\Ha}%
\lambda\left(A\cap\bigcup_{j=1}^IK_j\cap\bigcap_{n\le k} H_n^c\right)\\
&\le
\lim_k\sup_{B\in\Ha_1}\lambda\left(B\cap\bigcap_{n\le k} H_n^c\right)\\
&=
0
\end{align*}
which proves \eqref{limit}. 
\end{proof}

For the next result, define the $\lambda$-completion of $\A$ as follows
\begin{equation}
\label{A(l)}
\A(\lambda)=\left\{B\subset \Omega:
\inf_{\{A,A'\in\A:A\subset B\subset A'\}}\lambda(A'\backslash A)=0
\right\}
\end{equation}
It is clear that $\lambda$ admits exactly one extension to $\A(\lambda)$ defined by
letting 
\begin{equation}
\label{bar l}
\bar\lambda(B)
=
\sup_{\{A\in\A:A\subset B\}}\lambda(A)
=
\inf_{\{A'\in\A:B\subset A'\}}\lambda(A')
\qquad B\in\A(\lambda)
\end{equation}

Finite additivity often emerges upon taking extensions of a countably additive set function.
The following Corollary examines one such situation and establishes countable
additivity holds at least locally along some sequence.

\begin{corollary}
\label{cor bar l}
Let $\Bl=\A(\lambda)\backslash\A$ be non empty. There exists a disjoint sequence 
$\seqn A$ in $\A$ such that $\bigcup_nA_n\in\A(\lambda)$ and
\begin{equation}
\label{bar lambda}
\bar\lambda(B)=\sum_n\bar\lambda(B\cap A_n)
\qquad
B\in\A(\lambda)
\end{equation}
\end{corollary}

\begin{proof}
Choose
\begin{align*}
\Ha=\{H\in\A:H\subset B\text{ for some }B\in\Bl\}
\end{align*}
in Corollary \ref{cor drewnowski}. Then $\Bl\subset\A_\Ha$.
Extract the sequence $\seqn A$ from the sequence $\seqn H$ of Corollary 
\ref{cor drewnowski} by letting $A_n=H_n\backslash\bigcup_{j<n}H_j$ and
observe that $A_n\in\Ha$. By \eqref{limit} we obtain that
$\bar\lambda(B)=\sum_n\bar\lambda(B\cap A_n)$ for each $B\in\Bl$.
Observe that $\Bl$ is closed with respect to complementation and thus
\begin{align}
\label{eq}
\inf_{\{A'\in\A:\bigcup_nA_n\subset A'\}}\lambda(A')
\le
\lambda(\Omega)
=
\bar\lambda(B)+\bar\lambda(B^c)
=
\sum_n\lambda(A_n)
\le
\sup_{\{A\in\A:A\subset\bigcup_nA_n\}}\lambda(A)
\end{align}
which proves that $\bigcup_nA_n\in\A(\lambda)$. But then
$\bar\lambda(B)
\ge
\bar\lambda(B\cap\bigcup_nA_n)
\ge
\sum_n\bar\lambda(B\cap A_n)$ 
for each $B\in\Bl$. Applying this conclusion to $B\in\Bl$ and its complement
and exploiting \eqref{eq} one concludes that \eqref{bar lambda} necessarily
holds.
\end{proof}

Another possible development of Theorem \ref{th halmos savage} is the 
following finitely additive version of a theorem of Yan \cite[Theorem 2, p. 220]{yan} 
which is well known in stochastic analysis and mathematical finance:

\begin{corollary}[Yan]
\label{cor yan}
Let $\K\subset L^1(\lambda)$ be convex with $0\in\K$, write $\C=\K-\Sim(\A)_+$ 
and denote by $\overline{\C}$ the closure of $\C$ in $L^1(\lambda)$.
The following are equivalent: 
\begin{enumerate}[(i)]
\item\label{eta f} 
for each $f\in L^1(\lambda)_+$ with $\lambda(f)>0$ 
there exists $\eta>0$ such that $\eta f\notin\overline{\C}$;
\item\label{d A} 
for each $A\in\A$ with $\lambda(A)>0$ there exists 
$d>0$ such that $d\set A\notin\overline{\C}$;
\item\label{m} 
there exists a finitely additive probability $P$ on $\A$ such that 
\begin{enumerate}[(a)]
\item
$\K\subset L^1(P)$ and $\sup_{k\in\K}P(k)<\infty$,
\item
$\sup_{\{A\in\A:\lambda(A)>0\}}P(A)/\lambda(A)<\infty$ and
\item
$P(A)=0$ if and only if $\lambda(A)=0$.
\end{enumerate}
\end{enumerate}
\end{corollary}

\begin{proof}
The implication \imply{eta f}{d A} is obvious. If $A$ and $d$ are as in \iref{d A} 
there exists a continuous linear functional $\phi^A$ on $L^1(\lambda)$ and $a$
and $b$ such that
\begin{equation*}
\sup_{x\in\cl\C{}}\phi^A(x)<a<b<\phi^A(d1_A)
\end{equation*}
Given that $\mathcal C$ contains the convex cone $-\Sim(\A)_+$, that 
$\Sim(\A)_+$ is dense in $L^1(\lambda)_+$ and that $\phi^A$ is continuous, 
we conclude that $\sup_{f\in L^1(\lambda)_+}\phi^A(-f)<\infty$ i.e. that 
$\phi^A\ge0$. It follows from \cite[Theorem 2]{JMAA} that $\phi^A$ admits 
the representation $\phi^A(f)=\mu^A(f)$ for some $\mu^A\in ba(\lambda)_+$. 
Moreover,
\begin{align*}
\sup_{\{B\in\A:\lambda(B)>0\}}\mu^A(B)/\lambda(B)
\le
\sup_{\{f\in L^1(\lambda):\norm f\le1\}}\phi^A(f)
=
\norm{\phi^A}
<
\infty
\end{align*}
and $\sup_{h\in\C}\mu^A(f)<a<b<d\mu^A(A)$ so that $\mu^A$ meets (\textit a)
and (\textit b) above. The inclusion 
$0\in\C$ implies $a>0$ so that $\mu^A(A)>0$. By normalization we can assume 
$\norm{\phi^A}\vee a\le1$. The collection $\M=\{\mu^A:A\in\A,\ \lambda(A)>0\}$ so 
obtained is dominated by $\lambda$ and therefore by some $m\in\Av\M$, by 
Theorem \ref{th halmos savage}. Thus $m\le\lambda$ and 
$\sup_{h\in\C}m(h)\le 1$. If $A\in\A$ and $\lambda(A)>0$ then 
$m\gg\mu^A$ implies $m(A)>0$. The implication \imply{d A}{m} follows upon
letting $P$ be the finitely additive probability obtained from $m$ by normalization. 
Let $P$ be as in \iref{m} so that $L^1(\lambda)\subset L^1(P)$, by (\textit b). 
If $f\in L^1(\lambda)_+$ and $\lambda(f)>0$ then $f\wedge n$ converges 
to $f$ in $L^1(\lambda)$ \cite[III.3.6]{bible} so that we can assume that $f$ is 
bounded. Then, by \cite[4.5.7 and 4.5.8]{rao}, there exists an increasing sequence 
$\seqn f$ in $\Sim(\A)$ with $0\le f_n\le f$ such that $f_n$ converges to $f$ in 
$L^1(\lambda)$ and therefore in $L^1(P)$ too. For $n$ large enough, then, 
$\lambda(f_n)>0$ and, $f_n$ being positive and simple, $P(f_n)>0$. But 
then $P(f)=\lim_nP(f_n)>0$ so that $\eta f$ cannot be an element of 
$\overline\C$ for all $\eta>0$ as $\sup_{h\in\overline\C}P(h)<\infty$.
\end{proof}

An application of Corollary \ref{cor yan} is obtained in \cite[Lemma 3.1]{BK}. 
Corollary \ref{cor drewnowski} also provides a finitely additive version of a useful 
result of Mukherjee and Summers \cite[Lemma 3]{mukherjee summers}, illustrating 
the countable structure of the atoms of an additive set function%
\footnote{I am in debt with an anonymous referee for calling my attention on 
this paper.}.

\begin{corollary}[Mukherjee and Summers]
\label{cor mukherjee}
Let $\lambda$ have atoms. There exists a countable, pairwise disjoint collection 
$G_1,G_2,\ldots$ of $\lambda$-atoms of $\A$ such that for any $\lambda$-atom 
$B\in\A$ there exists $n\in\N$ such that $\lambda(B\Delta G_n)=0$.
\end{corollary}

\begin{proof}
Apply Corollary \ref{cor drewnowski} with $\Ha$ the collection of all $\lambda$-atoms 
of $\A$. Let $\seqn H$ be the corresponding sequence in $\Ha$ and put
$G_n=H_n\backslash\bigcup_{i<n}H_i$. Upon passing to a subsequence if necessary 
we may assume $\lambda(G_n)>0$ so that $G_n\in\Ha$ for each $n\in\N$. If 
$B\in\Ha$ it follows from \eqref{limit} that $\lambda(B\cap G_n)>0$ for some 
$n$. Given that $B$ and $G_n$ are atoms then 
$\lambda(B\backslash G_n)=\lambda(G_n\backslash B)=0$.
\end{proof}

\end{document}